\documentclass[12pt]{article}
\usepackage{amsfonts}
\usepackage{amssymb}
\usepackage{bbm}
\usepackage{amscd}
\usepackage{mathrsfs}
\usepackage[flushleft]{threeparttable}
\usepackage{amsmath,amssymb, amsthm}
\usepackage{graphicx}
\usepackage{color}
\usepackage{authblk}
\usepackage{subcaption}
\usepackage{indentfirst,latexsym,amssymb}
\usepackage[bookmarks=false]{hyperref}
\usepackage{tikz}

\newtheorem{theorem}{Theorem}[section]
\newtheorem{pro}[theorem]{Proposition}
\newtheorem{conj}[theorem]{Conjecture}
\newtheorem{lem}[theorem]{Lemma}

\newtheorem{coro}[theorem]{Corollary}

\theoremstyle{definition}

\newtheorem{exam}[theorem]{Exapmle}

\def\00{\mathbf{0}}

\def\CS{{\mathrm{CS}}}

\usepackage{xcolor}
\usepackage[normalem]{ulem}

\begin{document}

\title{On subgroup perfect codes in Cayley sum graphs}
\author{Jun-Yang Zhang}
\affil{School of Mathematical Sciences, Chongqing Normal University\\ Chongqing 401331, People's Republic of China}
\date{}

\openup 0.5\jot
\maketitle

 \footnotetext{E-mail address: jyzhang@cqnu.edu.cn (Jun-Yang Zhang)}
\begin{abstract}
A perfect code $C$ in a graph $\Gamma$ is an independent set of vertices of $\Gamma$ such that every vertex outside of $C$ is adjacent to a unique vertex in $C$, and a total perfect code $C$ in $\Gamma$ is a set of vertices of  $\Gamma$ such that every vertex of $\Gamma$ is adjacent to a unique vertex in $C$. Let $G$ be a finite group and $X$ a normal subset of $G$. The Cayley sum graph $\CS(G,X)$ of $G$ with the connection set $X$ is the graph with vertex set $G$ and two vertices $g$ and $h$ being adjacent if and only if $gh\in X$ and $g\neq h$.
In this paper, we give some necessary conditions of a subgroup of a given group being a (total) perfect code in a Cayley sum graph of the group. As applications, the Cayley sum graphs of some families of groups which admit a subgroup as a (total) perfect code are classified.


\end{abstract}
\section{Introduction}
\label{sec:int}

Throughout the paper, all groups are finite with identity element denoted by $1$, and all graphs are finite, undirected and simple. Cayley sum graph, which is also called addition graph \cite{CGW2003}, addition Cayley graph \cite{GLS2009,L2010,SGS2011} and sum graph \cite{C1989}, is a variation of the well-studied Cayley graph. The concept of Cayley sum graphs was at first only for abelian groups and then generalized over arbitrary groups in \cite{AT2016}. Let $G$ be a group and $X$ a normal subset of $G$ (that is, $g^{-1}Xg=X$ for all $g\in G$). The \emph{Cayley sum graph} $\CS(G,X)$ of $G$ with \emph{connection set} $X$ is the graph with vertex set $G$ and two vertices $g$ and $h$ being adjacent if and only if $gh\in X$ and $g\neq h$. Note that this definition differs from that in \cite{AT2016} where loops in graphs are allowed and so $x\neq y$ is not required. An element $a$ of $G$ is called a \emph{square} if $a=g^2$ for some $g\in G$ and a \emph{nonsquare} if otherwise. A subset of $G$ is said to be \emph{square-free} if it contains no square of $G$. It is obvious that the neighbourhood of a vertex $g$ is $Xg^{-1}$ if $g^{2}\notin X$ and $(X\setminus\{g^2\})g^{-1}$ if $g^{2}\in X$. Therefore $\CS(G,X)$ is a regular graph if and only if $X$ is square-free in $G$.

\smallskip
Let $\Gamma$ be a graph with vertex set $V(\Gamma)$ and edge set $E(\Gamma)$. Let $C$ be a subset of $V(\Gamma)$. If $C$ is an independent set of $\Gamma$ and every vertex in $V(\Gamma) \setminus C$ has exactly one neighbor in $C$, then $C$ is called a \emph{perfect code} in $\Gamma$. If every vertex of $\Gamma$ has exactly one neighbor in $C$, then $C$ is called a \emph{total perfect code} in $\Gamma$. It is obvious that a total perfect code in $\Gamma$ induces a matching in $\Gamma$ and therefore has even cardinality.
In graph theory, a perfect code in a graph is also called an \emph{efficient dominating set} \cite{DeS} or \emph{independent perfect dominating set }\cite{Le}, and a total perfect code is called an \emph{efficient open dominating set} \cite{HHS}.

\smallskip
The concept of perfect codes in graphs were firstly introduced by Biggs \cite{Big} as a generalization of the classical notions of perfect Hamming- and Lee-error-correcting codes. Perfect codes in Cayley graphs have received considerable attention see \cite[Section 1]{HXZ18} for a brief survey and \cite{CWX2020,DSLW16, FHZ, MWWZ19, Ta13, Zhou2016, Z15,ZZ2021} for a few recent papers. In particular, perfect codes in Cayley graphs which are subgroups of the underlying groups are especially interesting since they are generalizations of perfect linear codes \cite{Va75} in the classical setting.

\smallskip
Let $\CS(G,X)$ be a Cayley sum graph of $G$ with the connection set $X$. We use $\mathcal{P}(G,X)$ ($\mathcal{T}(G,X)$) to denote the collection of all subsets of $G$ which are perfect codes (total perfect codes) in $\CS(G,X)$. A subgroup $H$ of $G$ is called a \emph{subgroup perfect code} (\emph{subgroup total perfect code}) in $\CS(G,X)$ if $H\in \mathcal{P}(G,X)$ ($H\in \mathcal{T}(G,X)$). Note that $\{1\}$ is a perfect code of $\CS(G,G\setminus\{1\})$. A subgroup perfect code $H$ of $\CS(G,X)$ is said to be \emph{nontrivial} if $H\neq \{1\}$. Very recently, subgroup perfect codes in regular Cayley sum graphs of abelian groups were studied in \cite{MFW2020, MWY2022}. In this paper, we study subgroup (total) perfect codes in Cayley sum graphs of groups where the graphs are not necessary regular and groups are not necessary abelian. The results in this paper include a necessary and sufficient condition and a few necessary conditions of a subgroup of a given group being a (total) perfect code in a Cayley sum graph of the group. These results are used to study (total) perfect codes in a Cayley sum graphs of abelian groups, dihedral groups and one-dimensional  affine groups.

\smallskip
Now we describe the results of this paper in detail as follows.
In the next section, we prove a proposition which characterizes the relationship between normal subgroup perfect codes and normal subgroup total perfect codes in Cayley sum graphs.
In Section \ref{sec:pc}, we focus on the study of subgroup perfect codes in Cayley sum graphs. After showing an easy necessary and sufficient condition, we mainly prove a few necessary conditions of a subgroup $H$ of a given group $G$ being a perfect code in a Cayley sum graph of $G$. Using those necessary conditions, we prove that if the core of $H$ in $G$ is not contained in the center of $G$, then there is no connected Cayley sum graph of $G$ admitting $H$ as a perfect code.
In Section \ref{sec:tpc}, we develop the theory for perfect codes obtained in Section \ref{sec:pc} in parallel to that for total perfect codes.
In Section \ref{sec:groups}, we study (total) perfect codes in Cayley sum graphs of some spacial families of groups, including classifications of (total) perfect codes in connected Cayley sum graphs of abelian groups, dihedral groups and  one-dimensional affine groups respectively. In particular, we give a simple proof of a main result in \cite{MWY2022} which characterizes subgroup perfect codes in regular Cayley sum graphs of abelian groups.
\section{Preliminaries}
\label{sec:prel}

In \cite{AT2016}, a necessary and sufficient condition of a Cayley sum graph being connected was given. Note that this result also holds for Cayley sum graphs in the present paper as these graphs are exactly  obtained from the Cayley sum graphs in \cite{AT2016} by removing all loops.
This result is stated as follows.
\begin{lem}[\cite{AT2016}]
\label{connected}
Let $X$ be a normal subset of a group $G$. Then the Cayley sum graph $\CS(G,X)$ is connected if and only if $G=\langle X\rangle$ and
$|G:\langle X^{-1}X\rangle|\leq2$.
\end{lem}
Recall that the center $Z(G)$ of a group $G$ is a subgroup of $G$ consisting of all elements $x\in G$ such that $xg=gx$ for all $g\in G$. The following lemma characterizes the relationship between normal subgroup perfect codes and normal subgroup total perfect codes in Cayley sum graphs.
\begin{pro}
\label{relation1}
Let $G$ be a group and $H$ a normal subgroup of $G$. Then $H$ is a total perfect code of some Cayley sum graph of $G$ if and only if $H$ is perfect code of some Cayley sum graph of $G$ and $Z(G)$ contains a nonsquare of $H$.
\end{pro}
\begin{proof}
Let $H\in\mathcal{T}(G,Y)$ for some normal subset $Y$  of $G$. Then every $g\in G$ is adjacent to a unique vertex $h\in H$ in $\CS(G,Y)$. Let $z$ be the unique element in $H$ adjacent to $1$. Then $H\cap Y=\{z\}$ and $z$ is a nonsquare of $H$. Since both $H$ and $Y$ are normal in $G$, we have $z^G\subseteq H\cap Y$ and therefore $z^g=z$ for all $g\in G$. Thus $z\in Z(G)$.
Set $X=Y\setminus \{z\}$. Since $Y$ is normal in $G$ and $z\in Z(G)$, we have that $X$ is normal in $G$. Note that $H\cap X=\emptyset$. Therefore $H$ is an independent set of $\CS(G,X)$. Let $a$ be an arbitrary element in $G\setminus H$ and $h$ the unique vertex in $H$ adjacent to $a$ in $\CS(G,Y)$. Then $h$ is the unique element in $H$ satisfying $ab\in X$. In other words, $h$ is the unique vertex in $H$ adjacent to $a$ in $\CS(G,X)$. It follows that $H\in\mathcal{P}(G,X)$.

Let $H\in\mathcal{P}(G,X)$ for some normal subset $X$ of $G$, $z\in Z(G)\cap H$ and $z$ be a nonsquare of $H$. Set $Y=X\cup\{z\}$. Then $Y$ is normal in $G$ and it straightforward to check that $H\in\mathcal{T}(G,X)$.
\end{proof}
\section{Perfect codes}
\label{sec:pc}

In this section, we at first confirm a few simple facts on perfect codes in Cayley sum graphs and then prove two theorems which characterize those codes more deeply.

\smallskip
The first lemma gives a necessary and  sufficient condition of a subgroup of a given group being a perfect code in a Cayley sum graph of the group.
\begin{lem}
\label{iff}
Let $G$ be a group and $H$ a subgroup of $G$. Then $H$ is a perfect code of some Cayley sum graph of $G$ if and only if $G$ has a normal subset $X$ such that $X\cup\{1\}$ is a left transversal of $H$ in $G$.
\end{lem}
\begin{proof}
$\Rightarrow$) Let $H\in\mathcal{P}(G,X)$ for some normal subset $X$  of $G$. Then $H$ is an independent set of $\CS(G,X)$, and every $g\in G\setminus H$ is adjacent to a unique vertex $h\in H$. It follows that $H\cap X=\emptyset$ and every $g\in G\setminus H$ can be uniquely written as $g=xh^{-1}$ for some $x\in X$ and $h\in H$. Therefore $X\cup\{1\}$ is a left transversal of $H$ in $G$.

$\Leftarrow$) Let $X$ be a normal subset of $G$ such that $X\cup\{1\}$ is a left transversal of $H$ in $G$. Then $H\cap X=\emptyset$ and every $g\in G\setminus H$ can be uniquely written as $g=xh^{-1}$ for some $x\in X$ and $h\in H$. Therefore $H$ is an independent set of the Cayley sum graph $\CS(G,X)$ and every $g\in G\setminus H$ is adjacent to a unique vertex $h\in H$. Thus $H\in\mathcal{P}(G,X)$.
\end{proof}

Let $\sigma$ be an automorphism of $G$. We use $x^{\sigma}$ to denote the image of $x$ under $\sigma$ for all $x\in G$ and set $X^{\sigma}:=\{x^{\sigma}\mid x\in X\}$ for all subset $X$ of $G$. For each  $g\in G$, we write $x^g:=g^{-1}xg$ and $X^{g}:=\{x^{g}\mid x\in X\}$.
Since
\begin{equation*}
(X^{\sigma})^g=g^{-1}X^{\sigma}g=
\big((g^{\sigma^{-1}})^{-1}\big)^\sigma X^{\sigma}(g^{\sigma^{-1}})^\sigma
=\big((g^{\sigma^{-1}})^{-1}Xg^{\sigma^{-1}}\big)^\sigma=
(X^{g^{\sigma^{-1}}})^\sigma,
\end{equation*}
we conclude that $X$ is normal in $G$ if and only if $X^{\sigma}$ is normal in $G$. Obviously, if $X\cup\{1\}$ is a left transversal of $H$ in $G$, then $X^{\sigma}\cup\{1\}$ is a left transversal of $H^{\sigma}$ in $G$. Therefore Lemma \ref{iff} leads to the following result.
\begin{lem}
\label{auto}
Let $G$ be a group, $H$ a subgroup of $G$, $X$ a normal subset of $G$ and $\sigma$ an automorphism of $G$. Then $H\in\mathcal{P}(G,X)$ if and only if $H^{\sigma}\in\mathcal{P}(G,X^{\sigma})$. In particular, $H\in\mathcal{P}(G,X)$ if and only if $H^g\in \mathcal{P}(G,X)$ for a given $g\in G$.
\end{lem}
Unlike perfect codes in Cayley graph, the subgroup $H$ of $G$ being a perfect code of a Cayley sum graph $\CS(G,X)$ can not guarantee that a coset $Hg$ is a perfect code of $\CS(G,X)$. Actually, $Hg$ is not necessary an independent set of $\CS(G,X)$ when $H$ is. However, we have the follow lemma.
\begin{lem}
\label{coset}
Let $G$ be a group, $H$ a subgroup of $G$, $X$ a normal subset of $G$ and $b$ an involution of $G$ such that $H^b=H$.
Then $H\in \mathcal{P}(G,X)$ if and only if $Hb\in \mathcal{P}(G,X)$.
\end{lem}
\begin{proof}
$\Rightarrow$) Let $H\in\mathcal{P}(G,X)$. By Lemma \ref{iff}, $X\cup\{1\}$ is a left transversal of $H$ in $G$. Since $b$ is an involution and $H^b=H$, we get $(Hb)(Hb)=HH^b=H$. Therefore $Hb$ is an independent set of $\CS(G,X)$ as $(Hb)(Hb)\cap X=H\cap X=\emptyset$. Now let $g$ be an arbitrary element in $G\setminus Hb$. Then $gb\notin H$. Since $X\cup\{1\}$ is a left transversal of $H$ in $G$, we get $|gbH\cap X|=1$. Therefore there exists a unique $h\in H$ such that $gbh\in X$, that is, $bh$ is the unique vertex in $Hb$ ($=bH$) adjacent to $g$ in $\CS(G,X)$. It follows that $Hb\in\mathcal{P}(G,X)$.

\smallskip
$\Leftarrow$) Let $Hb\in\mathcal{P}(G,X)$. Then $Hb$ is an independent set of $\CS(G,X)$ and every vertex in $G\setminus Hb$ is adjacent to a unique vertex in $H$. Since $b$ is an involution and $H^b=H$, we have $H\cap X=(Hb)(Hb)\cap X=\emptyset$. Therefore $H$ is an independent set of $\CS(G,X)$. Let $g$ be an arbitrary element in $G\setminus H$. Then $gb\notin Hb$. Let $hb$ be the unique vertex in $Hb$ adjacent to $gb$. Then $hb$ is the unique vertex in $Hb$ satisfying $gbhb\in X$. Set $h_1=h^b$. Then $h_1$ is a unique vertex in $H$ satisfying $gh_1\in X$. Therefore $H\in\mathcal{P}(G,X)$.
\end{proof}
We use $\dot\cup_{i=1}^{n}S_{i}$ to denote the union of the pair-wise disjoint sets $S_1,S_2,\ldots,S_n$. Let $G$ be a group and $H$ a subgroup of $G$. We use $|G:H|$ to denote the index of $H$ in $G$. The \emph{core} of $H$ in $G$ is the largest normal subgroup of $G$ contained in $H$. For each $g\in G$, we use $C_G(g)$ to denote the centralizer of $g$ in $G$. The following theorem gives a few necessary conditions of $H$ being a perfect code in a Cayley sum graph of $G$.
\begin{theorem}
\label{normal}
Let $G$ be a group, $X$ a normal subset of $G$, and $H$ a subgroup perfect code of the Cayley sum graph $\CS(G,X)$. Then the following statements hold.
\begin{enumerate}
  \item The core of $H$ in $G$ is contained in $C_{G}(x)$ for each $x\in X$.
  \item If $X=\dot\cup_{i=1}^{s}x_{i}^{G}$ and $H$ is normal in $G$, then $\frac{1}{|G:H|}+\sum_{i=1}^{s}\frac{1}{|C_{G}(x_i):H|}=1$.
  \item If $|X|>1$ and $H$ is normal in $G$, then $X$ is a union of at least two conjugacy classes of elements in $G$.
\end{enumerate}
\end{theorem}
\begin{proof}
(i) By Lemma \ref{iff}, $X\cup \{1\}$ is a left transversal of $H$ in $G$. Let $N$ be the core of $H$ in $G$. Suppose to the contrary that $N$ is not contained in $C_{G}(x)$ for some $x\in X$. Then $a^{-1}xa\neq x$ for some $a\in N$. Since $X$ is normal in $G$, we have $a^{-1}xa\in X$. However, since $N$ is the core of $H$ in $G$, we have $a^{-1}xaH=x(x^{-1}a^{-1}x)H=xH$. This contradicts the fact that $X\cup \{1\}$ is a left transversal of $H$ in $G$.

(ii) Since $X\cup \{1\}$ is a left transversal of $H$ in $G$, we have $1+|X|=|G:H|$. Since $X=\dot\cup_{i=1}^{s}x_{i}^{G}$, we get $1+\sum_{i=1}^{s}|x_{i}^{G}|=|G:H|$. Noting that $|x_{i}^{G}|=|G:C_{G}(x_i)|$ for every $i\in\{1,\ldots,s\}$, we obtain, $1+\sum_{i=1}^{s}|G:C_{G}(x_i)|=|G:H|$. By the conclusion of (i), we have that $H$ is contained in $C_{G}(x_i)$ for every $i\in\{1,\ldots,s\}$ as $H$ is normal in $G$. Therefore $|G:C_{G}(x_i)|=\frac{|G:H|}{|C_{G}(x_i):H|}$ and it follows that  $\frac{1}{|G:H|}+\sum_{i=1}^{s}\frac{1}{|C_{G}(x_i):H|}=1$.

(iii) If otherwise that $X=x^{G}$ for some $x\in G$, then (ii) implies  $\frac{1}{|G:H|}+\frac{1}{|C_{G}(x):H|}=1$. Since $H$ is normal in $G$, by (i) we have that $H$ is contained in $C_{G}(x)$.
Therefore $|G:H|=|G:C_{G}(x)||C_{G}(x):H|$. It follows that
$(1+|G:C_{G}(x)|)=|G:H|$. Since $|G:C_{G}(x)|$ divides $|G:H|$, we conclude that $|G:C_{G}(x)|=1$, contradicting to $|X>1|$.
\end{proof}
The following theorem provides a basis for denying some subgroups to be perfect codes in connected Cayley sum graphs.
\begin{theorem}
\label{core}
Let $G$ be a group and $H$ a subgroup of $G$. If the core of $H$ in $G$ is not contained in the center of $G$, then there is no connected Cayley sum graph of $G$ admitting $H$ as a perfect code. In particular, if $H$ is normal in $G$ and not contained in the center of $G$, then $H$ is not a perfect code of any connected Cayley sum graph of $G$.
\end{theorem}

\begin{proof}
Let $X$ be an arbitrary normal subset of $G$ such that the Cayley sum graph $\CS(G,X)$ is connected. By Lemma \ref{connected} we have $G=\langle X\rangle$. Since the core $N$ of $H$ in $G$ is not contained in the center of $G$, we conclude that $N$ is not contained in $C_{G}(x)$ for some $x\in X$.
By Theorem \ref{normal} (i), we have that $H\notin\mathcal{P}(G,X)$.

For the spacial case that $H$ is normal in $G$, the core of $H$ in $G$ is $H$ itself. Therefore $H$ is not a perfect code of any connected Cayley sum graph of $G$ provided that $H$ is not contained in the center of $G$.
\end{proof}
\section{Total perfect codes}
\label{sec:tpc}

In this section, we give some results on total perfect codes which are analogous to the results about perfect codes in Section \ref{sec:pc}.

\smallskip
Parallel to Lemma \ref{iff}, we obtain a necessary and  sufficient condition of a subgroup of a given group being a total perfect code in a Cayley sum graph of the group as follows.
\begin{lem}
\label{tiff}
Let $G$ be a group and $H$ a subgroup of $G$. Then $H$ is a total perfect code of some Cayley sum graph of $G$ if and only if $G$ contains a normal subset $Y$ such that $Y$ is a left transversal of $H$ in $G$ and the unique common element of $H$ and $Y$ is a nonsquare of $H$.
\end{lem}
\begin{proof}
$\Rightarrow$) Let $H\in\mathcal{T}(G,Y)$ for some normal subset $Y$  of $G$. Then every $g\in G$ is adjacent to a unique vertex $h\in H$, that is, every $g\in G$ can be uniquely written as $g=yh^{-1}$ for some $y\in Y$ and $h\in H$. Therefore $Y$ is a left transversal of $H$ in $G$. In particular, $|H\cap Y|=1$. Set $H\cap Y=\{z\}$. Since $H$ induces a matching of $\mathcal{T}(G,Y)$, we conclude that $z$ is a nonsquare of $H$.

$\Leftarrow$) Let $Y$ be a normal subset of $G$ such that $Y$ is a left transversal of $H$ in $G$  and the unique common element of $H$ and $Y$ is a nonsquare of $H$. Then every $g\in G\setminus H$ can be uniquely written as $g=yh^{-1}$ for some $y\in Y$ and $h\in H$. Therefore $H$ induces a matching of $\CS(G,Y)$ and every $g\in G\setminus H$ is adjacent to a unique vertex $h\in H$. Thus $H\in\mathcal{T}(G,Y)$.
\end{proof}

The following obvious result is the counterpart of Lemma \ref{auto}.
\begin{lem}
\label{tauto}
Let $G$ be a group, $H$ a subgroup of $G$, $Y$ a normal subset of $G$ and $\sigma$ an automorphism of $G$. Then $H\in\mathcal{T}(G,Y)$ if and only if $H^{\sigma}\in\mathcal{T}(G,Y^{\sigma})$. In particular, $H\in\mathcal{T}(G,Y)$ if and only if $H^g\in \mathcal{P}(G,Y)$ for a given $g\in G$.
\end{lem}
The following Lemma is akin to Lemma \ref{coset}. Its proof is omitted as it is similar to the proof of Lemma \ref{coset}.
\begin{lem}
\label{tcoset}
Let $G$ be a group, $H$ a subgroup of $G$, $Y$ a normal subset of $G$ and $b$ an involution of $G$ such that $H^b=H$.
Then the $H\in \mathcal{T}(G,X)$ if and only if $Hb\in \mathcal{T}(G,X)$.
\end{lem}
The following two theorems are the counterparts of Theorem \ref{normal} and \ref{core}. Note that Theorem \ref{coret} can be seen as a corollary of Theorem \ref{normalt}. Its proof is similar to that of Theorem \ref{core} and therefore omitted.
\begin{theorem}
\label{normalt}
Let $G$ be a group, $Y$ a normal subset of $G$, and $H$ a subgroup total perfect code of the Cayley sum graph $\CS(G,Y)$. Then the following statements hold.
\begin{enumerate}
  \item $H$ is contained in $C_{G}(z)$ where $z$ is the unique common element of $H$ and $X$.
  \item The core of $H$ in $G$ is contained in $C_{G}(y)$ for each $y\in Y$.
  \item If $Y=\dot\cup_{i=1}^{s}y_{i}^{G}$ and $H$ is normal in $G$, then $\sum_{i=1}^{s}\frac{1}{|C_{G}(y_i):H|}=1$.
  \item If $|Y|>1$ and $H$ is normal in $G$, then $Y$ is a union of at least two conjugacy classes of elements in $G$.
\end{enumerate}
\end{theorem}
\begin{proof}
(i) Let $z$ be the unique common element of $H$ and $Y$. Since $Y$ is normal in $G$, we have $y^h\in Y$ for every $h\in H$. Therefore $y^h\in H\cap Y$ as $z,h\in H$. By the uniqueness of $z$, we have $z^h=z$ and it follows that $H$ is contained in $C_{G}(z)$.

The proofs of (ii), (iii) and (iv) are omitted as they are similar to the proof of Theorem \ref{normal}.
\end{proof}
\begin{theorem}
\label{coret}
Let $G$ be a group and $H$ a subgroup of $G$. If the core of $H$ in $G$ is not contained in the center of $G$, then there is no connected Cayley sum graph of $G$ admitting $H$ as a total perfect code.
\end{theorem}
\section{For some spacial families of groups}
\label{sec:groups}
\subsection{Abelian groups}
The following result is obvious.
\begin{theorem}
\label{psgc}
Every subgroup $H$ of an abelian group $G$ is a perfect code of some Cayley sum graph $\CS(G,X)$ of $G$.
\end{theorem}
\begin{proof}
Since $G$ is abelian, every subset of $G$ is normal in $G$. Let $T$ be a left transversal of $H$ in $G$. Set $X=T\setminus H$. Then $X$ is a normal subset of $G$ and $X\cup\{1\}$ is a left transversal of $H$ in $G$. By Lemma \ref{iff}, $H$ is a perfect code of the Cayley sum graph $\CS(G,X)$.
\end{proof}
Recall that the Frattini subgroup $\Phi(G)$ of $G$ consists of elements $x$ of $G$ such that $G=\langle S,x\rangle$ leads to $G=\langle S\rangle$ for any subset $S$ of $G$. The Hall $2'$-subgroup of $G$ is an odd order subgroup of index a power of $2$. The following obvious result can be checked directly.
\begin{lem}
\label{abelsquare}
Let $G$ be an abelian group. Let $Q$ and $H$ be the Sylow $2$-subgroup and the Hall $2'$-subgroup of $G$ respectively. Then an element $g$ of $G$ is a square if and only if $g\in\Phi(Q)K$.
\end{lem}
The following result can be found in \cite{MWY2022}. Here we give it a short proof.
\begin{theorem}
[\cite{MWY2022}]
\label{psgcr}
Let $G$ be an abelian group and $H$ a subgroup of $G$. Then $H$ is a perfect code of some regular Cayley sum graph of $G$ if and only if $H$ contains a non-square of $G$ or $H=\Phi(Q)K$ where $Q$ and $K$ are respectively the Sylow $2$-subgroup and the Hall $2'$-subgroup of $G$.
\end{theorem}
\begin{proof}
$\Rightarrow$) Let $\CS(G,X)$ be a regular Cayley sum graph of $G$ and $H$ a perfect code of $\CS(G,X)$. Assume that every element of $H$ is a square of $G$. It suffices to show that $H=\Phi(Q)K$. By Lemma \ref{abelsquare}, $H$ is contained in $\Phi(Q)K$. Let $g$ be an arbitrary element of $\Phi(Q)K$. Since $H$ a perfect code of $\CS(G,X)$, it follows from Lemma \ref{iff} that $X\cup\{1\}$ is a left transversal of $H$ in $G$. Therefore $g$ can be uniquely written as $g=yh$ for some $y\in X\cup\{1\}$ and $h\in H$. Since $h,g\in \Phi(Q)K$, we get $y\in \Phi(Q)K$. Therefore $y$ is a square. On the other hand, $X$ is square-free as $\CS(G,X)$ is regular. It follows that $y=1$ and $g\in H$. Therefore $H=\Phi(Q)K$.

$\Leftarrow$) Let $X\cup \{1\}$ be a left transversal of $H$ in $G$. Then $X\cap H=\emptyset$. Since $G$ is abelian, $X$ is normal in $G$. If $H=\Phi(Q)K$, then Lemma \ref{abelsquare} implies that $X$ is square-free in $G$. It follows that $\CS(G,X)$ is a regular Cayley sum graph of $G$ and $H$ a perfect code of $\CS(G,X)$. In what follows we assume that $H$ contains a non-square, say $a$. Let $Z$ be a subset of $X$ consists of elements being squares of $G$. If $Z=\emptyset$, then we set $Y=X$. If $Z\neq\emptyset$, then we set $Y=aZ\cup(X\setminus Z)$. It is obvious that $Y\cup \{1\}$ is a left transversal of $H$ in $G$ and $Y$ is a square-free normal subset of $G$. Therefore $\CS(G,Y)$ is a regular Cayley sum graph of $G$ and $H$ a perfect code of $\CS(G,Y)$.
\end{proof}
The following result can be directly deduced from Proposition \ref{relation1} and Theorem \ref{psgc}.
\begin{theorem}
\label{tpsgc}
Every even order subgroup $H$ of an abelian group $G$ is a total perfect code of some Cayley sum graph of $G$.
\end{theorem}
By Proposition \ref{relation1} and Theorem \ref{psgcr}, we have the following result.
\begin{theorem}
\label{tpsgcr}
Let $G$ be an abelian group and $H$ an even order subgroup of $G$. Then $H$ is a total perfect code of some regular Cayley sum graph of $G$ if and only if $H$ contains a non-square of $G$ or $H=\Phi(Q)K$ where $Q$ and $K$ are respectively the Sylow $2$-subgroup and the Hall $2'$-subgroup of $G$.
\end{theorem}
\subsection{Dihedral groups}
Throughout this subsection. We use $\mathbb{D}_{2n}$ to denote the dihedral group order $2n$ which has a cyclic subgroup $\langle a\rangle$ of order $n$ and an involution $b\notin \langle a\rangle$ such that $a^b=a^{-1}$.
\begin{exam}
\label{D2nt}
The subset $X:=\{b,ab,a^2b,\ldots,a^{n-1}b\}$ is normal in $\mathbb{D}_{2n}$ and the Cayley sum graph $\Gamma:=\CS(\mathbb{D}_{2n},X)$ is isomorphic to the complete bipartite graph $K_{n,n}$. It is obvious that $\langle b\rangle$ is a total perfect code of $\Gamma$.
\end{exam}
\begin{exam}
\label{D4l}
Let $n=2\ell$ where $\ell$ is a positive integer. Then $b^{\mathbb{D}_{2n}}=\{b,a^2b,\ldots,a^{2\ell-2}b\}$ and $(ab)^{\mathbb{D}_{2n}}=\{ab,
a^3b,\ldots,a^{2\ell-1}b\}$. Set $Z=\{a^2,a^4,\ldots,a^{2\ell-2}\}$ and $Z'=\{a,a^3,\ldots,a^{2\ell-1}\}$. Then both $Z$ and $Z'$ are normal in $\mathbb{D}_{2n}$.
It is straightforward to check that $\langle ab\rangle$ is a perfect code of the Cayley sum graph
$\Gamma_0:=\CS(\mathbb{D}_{2n},b^{\mathbb{D}_{2n}}\cup Z)$ and $\langle b\rangle$ is a perfect code of the Cayley sum graph
$\Gamma_1:=\CS(\mathbb{D}_{2n},(ab)^{\mathbb{D}_{2n}}\cup Z)$. Moreover, $\langle b\rangle$ is a total perfect code of the Cayley sum graph
$\Gamma'_0:=\CS(\mathbb{D}_{2n},b^{\mathbb{D}_{2n}}\cup Z')$ and
$\langle ab\rangle$ is a total perfect code of the Cayley sum graph
$\Gamma'_1:=\CS(\mathbb{D}_{2n},(ab)^{\mathbb{D}_{2n}}\cup Z')$.
\end{exam}
\begin{exam}
\label{D2n}
Let $n=4k+2$ where $k$ is a positive integer. Then $b^{\mathbb{D}_{2n}}=\{b,a^2b,\ldots,a^{4k}b\}$ and $(ab)^{\mathbb{D}_{2n}}=\{ab,
a^3b,\ldots,a^{4k+1}b\}$.
Choose an inverse-closed left transversal $Z$ of $\langle a^{2k+1}\rangle$ in $\langle a\rangle$ (for example, let $Z=\{a,a^2,\ldots,a^k,a^{-1},a^{-2},\ldots,a^{-k},1\}$). Let
\begin{equation*}
 X=b^{\mathbb{D}_{2n}}\cup (Z\setminus\langle a^{2k+1}\rangle)
 ~\mbox{or}~(ab)^{\mathbb{D}_{2n}}\cup (Z\setminus \langle a^{2k+1}\rangle).
\end{equation*}
Then $X$ is a normal subset of $\mathbb{D}_{2n}$. It is straightforward to check that $\langle a^{2k+1}\rangle$ is a perfect code of the Cayley sum graph $\Gamma:=\CS(\mathbb{D}_{2n},X)$ and a total perfect code of the Cayley sum graph $\Gamma':=\CS(\mathbb{D}_{2n},X\cup\{a^{2k+1}\})$.
\end{exam}

\begin{theorem}
\label{pcd}
Every connected Cayley sum graph of $\mathbb{D}_{2n}$ has no nontrvial subgroup perfect code except the graphs $\Gamma_0$, $\Gamma_1$ in Example \ref{D4l} and $\Gamma$ in Example \ref{D2n}.
\end{theorem}
\begin{proof}
Let $\Sigma:=\CS(\mathbb{D}_{2n},X)$ be a connected Cayley sum graph of $\mathbb{D}_{2n}$. Then $X$ is a normal subset of $G$. By Lemma \ref{connected}, we have $\mathbb{D}_{2n}=\langle X\rangle$.
Therefore $X\cap \langle a\rangle b \neq\emptyset$. Suppose that $\Sigma$ has a nontrvial subgroup perfect code, say $H$. It suffices to prove that $\Sigma$ coincides with one of the graphs $\Gamma_0$, $\Gamma_1$ in Example \ref{D4l} and $\Gamma$ in Example \ref{D2n}.  By Lemma \ref{iff}, $X\cup\{1\}$ is a left transversal of $H$ in $\mathbb{D}_{2n}$ and therefore $|H|(|X|+1)=2n$. Since $|H|>1$, we have $|X|<n$. It follows that $\langle a\rangle b$ is not contained in $X$. Therefore $n$ is even and exactly one of the two conjugacy classes  $b^{\mathbb{D}_{2n}}$ and $(ab)^{\mathbb{D}_{2n}}$ is contained in $X$. Set $Y=X\cap \langle a\rangle b$. Then $Y=b^{\mathbb{D}_{2n}}~\hbox{or}~(ab)^{\mathbb{D}_{2n}}$. In particular, $|Y|=\frac{n}{2}$. It follows that $\frac{n}{2}<|X|+1\leq n$. Since $|H|(|X|+1)=2n$, we have $|H|=2~\hbox{or}~3$. By Theorem \ref{normal}, the core of $H$ in $\mathbb{D}_{2n}$ is contained in the center of $\mathbb{D}_{2n}$. Note that every odd order subgroup of $\mathbb{D}_{2n}$ is normal in $\mathbb{D}_{2n}$ but not contained in the center of $\mathbb{D}_{2n}$. Therefore $|H|\neq3$. It follows that $|H|=2$ and $|X|=n-1$. Write $n=2\ell$. Then $|Y|=\ell$ and $|X\setminus Y|=\ell-1$. The remainder proof is divided into two cases.

\textsf{Case 1.} $H\cap \langle a\rangle b\neq \emptyset$.

In this case, $H=\langle a^ib\rangle$ for some $i\in\{0,1,\ldots,n-1\}$. Set $Z=\{a^2,a^4,\ldots,a^{2\ell-2}\}$. If $i$ is odd, then $Z$ is the unique subset of $\mathbb{D}_{2n}$ such that $b^{\mathbb{D}_{2n}}\cup Z\cup\{1\}$ is  normal in $\mathbb{D}_{2n}$ and a left transversal of $H$ in $\mathbb{D}_{2n}$. Therefore $\Sigma$ is the graph $\Gamma_0$ in Example \ref{D4l}. Similarly, if $i$ is even, then $\Sigma$ is the graph $\Gamma_1$ in Example \ref{D4l}.

\textsf{Case 2.} $H\cap \langle a\rangle b=\emptyset$.

In this case, $H=\langle a^\ell\rangle$. Set $Z=X\setminus Y$. Since $X\cup\{1\}$ is a left transversal of $\langle a^\ell\rangle$ in $\mathbb{D}_{2n}$, we conclude that $Z\cup\{1\}$ is a left transversal of $\langle a^\ell\rangle$ in $\langle a\rangle$ and $Y\langle a^\ell\rangle=\langle a\rangle b$. Note that $Y=Ya^i$ provided $i$ is even. Therefore $Y\langle a^\ell\rangle\neq\langle a\rangle b$ if $\ell$ is even. It follows that $\ell$ is odd. Thus $\Sigma$ is the graph $\Gamma$ in Example \ref{D2n}.
\end{proof}
\begin{theorem}
\label{tpcd}
Every connected Cayley sum graph of $\mathbb{D}_{2n}$ has no subgroup total perfect code except the graphs $\Gamma$ in Example \ref{D2nt}, $\Gamma'_0$, $\Gamma'_1$ in Example \ref{D4l} and $\Gamma'$ in Example \ref{D2n}.
\end{theorem}
\begin{proof}
Let $\Sigma:=\CS(\mathbb{D}_{2n},X)$ be a connected Cayley sum graph admitting a subgroup total perfect code $H$.
Then $H$ is of even order, and $X$ is a normal subset of $\mathbb{D}_{2n}$ and a left transversal of $H$ in $\mathbb{D}_{2n}$.
It suffices to prove that $\Sigma$ coincides with one of the graphs $\Gamma$ in Example \ref{D2nt}, $\Gamma'_0$, $\Gamma'_1$ in Example \ref{D4l} and $\Gamma'$ in Example \ref{D2n}.
If $\langle a\rangle b\subseteq X$, then we have $|H|=2$ and $X=\langle a\rangle b$ as $2n=|H||X|$. It follows that $\Sigma$ is the graph $\Gamma$ in Example \ref{D2nt}.
 If the unique common element of $X$ and $H$ is contained in $\langle a\rangle$, then $H$ is subgroup perfect code of the Cayley sum graph $\CS(\mathbb{D}_{2n},X\setminus \langle a\rangle)$. By Theorem \ref{pcd}, we have that $\CS(\mathbb{D}_{2n},X\setminus \langle a\rangle)$ is the graph $\Gamma$ in Example \ref{D2n}. Therefore $\Sigma$ is the graph $\Gamma'$ in Example \ref{D2n}. Now we assume that $\langle a\rangle b$ is not a subset of $X$ and the unique common element of $X$ and $H$ is not contained in $\langle a\rangle$. Then $X\cap H=\langle a^ib\rangle$ for some $i\in \{1,\ldots,n\}$ and
 $(a^ib)^{\mathbb{D}_{2n}}\neq \langle a\rangle b$. It follows that $n$ is even, $X\cap\langle a\rangle b=(a^ib)^{\mathbb{D}_{2n}}$, $|(a^ib)^{\mathbb{D}_{2n}}|=\frac{n}{2}$ and $\langle(a^ib)^{\mathbb{D}_{2n}}\rangle\neq\mathbb{D}_{2n}$. Since $\Sigma$ is connected, we have $\mathbb{D}_{2n}=\langle X\rangle$. Therefore $X\setminus (a^ib)^{\mathbb{D}_{2n}}\neq\emptyset$ and so $|X|>\frac{n}{2}$. Since $|H|$ is even and $2n=|H||X|$, we have$|H|=2$ and $|X|=n$. In particular, $H=\langle a^ib\rangle$. Therefore $(a^ib)^{\mathbb{D}_{2n}}H=\{1,a^2,\ldots,a^{2\ell-2}\}\cup (a^ib)^{\mathbb{D}_{2n}}$ and it follows that $X\setminus (a^ib)^{\mathbb{D}_{2n}}=\{a,a^3,\ldots,a^{2\ell-1}\}$. Thus $\Sigma$ is the graph $\Gamma'_0$ or $\Gamma'_1$ in Example \ref{D4l}.
\end{proof}
\subsection{One-dimensional  affine groups}
Let $\mathbb{F}_q$ be the finite field of order $q$, $\mathrm{AGL}_1(q)$ the one-dimensional affine group over $\mathbb{F}_q$ where $q$ is a prime power. It is well known that  $\mathrm{AGL}_1(q)$ is a Frobenius group with an elementary abelian Frobenius kernel of order $q$ and cyclic Frobenius complements of order $q-1$. Throughout this subsection, let $G=\mathrm{AGL}_1(q)$, and use $K$ and $C$ to denote the Frobenius kernel and a Frobenius complement of $G$ respectively. By the Frobenius partitions (see \cite[Page 79]{KS2004}) of Frobenius groups, we have the following lemma.
\begin{lem}
\label{partition}
$G=K\cup(\cup_{x\in K}C^x)$.
\end{lem}
We aim to classify Cayley sum graphs of one-dimensional affine groups admitting a subgroup (total) perfect code. Before doing that, we need prove the following lemma.
\begin{lem}
\label{length}
Let $Y$ be a conjugacy class of $G$ not containing $1$. Then either $Y=K\setminus\{1\}$ or
$|Y|=q$ and $|Y\cap C^{g}|=1$ for all $g\in G$.
\end{lem}
\begin{proof}
Take $y\in Y$. Then $Y=y^G$. It suffices to show that $Y=K\setminus\{1\}$ if $y\in K\setminus\{1\}$, and $|Y|=q$ and $|Y\cap C^{g}|=1$ for all $g\in G$ if $y\notin K$.

Firstly, we assume $y\in K\setminus\{1\}$. Since $K$ is normal in $G$, we conclude that $Y\subseteq K\setminus\{1\}$. Since $G$ is a Frobenius group with the Frobenius kernel $K$ and a Frobenius complement $C$, we have $C^y\cap C=1$. In particular $cy\neq yc$ for all $c\in C\setminus\{1\}$. Therefore $y^{c_{1}}\neq y^{c_{2}}$ for each pair of distinct elements $c_1,c_2\in C$. This leads to $|Y|\geq|C|=q-1$. Since $Y\subseteq K\setminus\{1\}$ and $|K\setminus\{1\}|=q-1$, it follows that
$Y=K\setminus\{1\}$.

Now we assume that $y\notin K$. By Lemma \ref{partition}, $y$ is an nonidentity element contained in a Frobenius complement of $G$. Without loss of generality, we assume that $y\in C\setminus\{1\}$.
Then $y^{x_1}\neq y^{x_1}$ for each pair of distinct elements $x_1,x_2\in K$ as $C^{x_1x^{-1}}\cap C=\{1\}$. Therefore $|y^{K}|=|K|=q$. Since $G=CK$ and $C$ is abelian, we have $Y=y^{CK}=y^{K}$ and it follows that $|Y|=q$. For every $g\in G$, we have $y^g\in C^g$. If there exists $h\in G$ satisfying $y^h\in C^g$, then $y\in C^{gh^{-1}}\cap C$ and it follows that $gh^{-1}\in C$. Therefore $y^{gh^{-1}}=y$. This implies that $y^{g}=y^h$. Therefore $y^g$ is the unique element of $Y\cap C^g$, that is, $|Y\cap C^{g}|=1$.
\end{proof}
The following two theorems give a classification of Cayley sum graphs of one-dimensional affine groups which admit a subgroup (total) perfect code.
\begin{theorem}
\label{Frobenius}
Let $C=\langle c\rangle$ and $q-1=st$ with $t>1$. Let $\{a_0,a_1,\ldots,a_{s-1}\}$ be a left transversal of $\langle c^s\rangle$ in $C$ with $a_0\in\langle c^s\rangle\setminus\{1\}$. Set $X:=(\cup_{i=1}^{s-1}a_{i}^{G})\cup(K\setminus\{1\})$ and $Y:=\cup_{i=0}^{s-1}a_{i}^{G}$.
Then $\langle c^s\rangle$ is a perfect code of $\CS(G,X)$.
Moreover, if $c_0$ is a nonsquare of $\langle c^s\rangle$,
then $\langle c^s\rangle$ is a total perfect code of $\CS(G,Y)$.
\end{theorem}
\begin{proof}
Since $\{a_0,a_1,\ldots,a_{s-1}\}$ is a left transversal of
$\langle c^s\rangle$ in $C$ and $a_0\in\langle c^s\rangle\setminus\{1\}$, we have that $a_i\neq1$ for every $i\in\{0,1,\ldots,s-1\}$. By Lemma \ref{length}, we have $|a_i^{G}|=q$ and $a_i^{G}\cap C=\{a_i\}$. Therefore $a_j^{G}\cap a_k^{G}=\emptyset$ for each pair of distinct elements $a_j$ and $a_k$. It follows that $|Y|=|\cup_{i=0}^{s-1}a_{i}^{G}|=\sum_{i=0}^{s-1}|a_{i}^{G}|=sq$. Since $q-1=st$ and $|\langle c\rangle|=q-1$, we have
$|Y||\langle c^s\rangle|=sqt=q(q-1)=|G|$. Take $a_i^{g},a_j^{h}\in Y$ such that $a_i^{g}\langle c^s\rangle=a_j^{h}\langle c^s\rangle$. Since $(a_i^{g})^{-1}a_{j}^{h}=[g,a_i][ha_i,a_j^{-1}]a_ia_j^{-1}$, we have  $[g,a_i][ha_i,a_j^{-1}]a_ia_j^{-1}\in \langle c^s\rangle$. By Lemma \ref{partition}, we have $[g,a_i][ha_i,a_j^{-1}]\in K$. It follows that $[g,a_i][ha_i,a_j^{-1}]=1$ and $a_ia_j^{-1}\in\langle c^s\rangle$. Since $\{a_0,a_1,\ldots,a_{s-1}\}$ is a left transversal of
$\langle c^s\rangle$ in $C$, we have $a_i=a_j$ and therefore $[g,a_i][ha_i,a_i^{-1}]=1$. Since $[g,a_i][ha_i,a_i^{-1}]=
g^{-1}a_i^{-1}ga_ia_i^{-1}h^{-1}a_iha_ia_i^{-1}=(a_i^g)^{-1}a_i^h$, we get
$a_i^g=a_i^h$, that is, $a_i^g=a_j^h$. Now we have proved that $|Y||\langle c^s\rangle|=|G|$ and $a_i^{g}\langle c^s\rangle=a_j^{h}\langle c^s\rangle$ implies $a_i^g=a_j^h$ for $a_i^{g},a_j^{h}\in Y$. Therefore $Y$ is a left transversal of $\langle c^s\rangle$ in $G$. By Lemma \ref{tiff}, $\langle c^s\rangle$ is a total perfect code of $\CS(G,Y)$ if $a_0$ is a non-square of $\langle c^s\rangle$.

It remains to prove that $\langle c^s\rangle$ is a perfect code of $\CS(G,X)$. Since $K$ is normal in $G$, we have that $K\langle c^s\rangle$ is a subgroup of $G$. By Lemma \ref{partition}, we have $[g,a_0^{-1}]\in K$ for all $g\in G$. Therefore $a_0^g=[g,a_0^{-1}]a_0\in K\langle c^s\rangle$ as $a_0\in\langle c^s\rangle$. It follows that
$a_0^G\langle c^s\rangle\subseteq K\langle c^s\rangle$. Since
 $|a_0^G\langle c^s\rangle|=|a_0^G||\langle c^s\rangle|=qt=|K\langle c^s\rangle|$, we conclude that $a_0^G\langle c^s\rangle=K\langle c^s\rangle$. Note that $X\cup\{1\}=(Y\setminus a_0^G)\cup K$. Therefore
 $(X\cup\{1\})\langle c^s\rangle=Y\langle c^s\rangle=G$. Since $|X\cup\{1\}|=|Y|$ and $Y$ is a left transversal of $\langle c^s\rangle$ in $G$, it follows that $X\cup\{1\}$ is a left transversal of $\langle c^s\rangle$ in $G$. By Lemma \ref{iff}, $\langle c^s\rangle$ is a perfect code of $\CS(G,X)$.
\end{proof}
\begin{theorem}
Every Cayley sum graph of $\mathrm{AGL}_1(q)$ has no nontrivial subgroup perfect code except the graphs $\CS(G,X)$ in Theorem \ref{Frobenius} and no subgroup total perfect code except the graphs $\CS(G,Y)$ in Theorem \ref{Frobenius}.
\end{theorem}
\begin{proof}
Suppose that $\CS(G,X)$ is a Cayley sum graph of $G$ admitting a subgroup perfect code $H$. Since $X$ is a normal subset of $G$, it is a disjoint union of conjugacy classes of $G$. By Lemma \ref{length}, either $(K\setminus\{1\})\subseteq X$ or $|X|=\ell q$ for some integer $\ell$. By Lemma \ref{iff}, $\{1\}\cup X$ is a left transversal of $H$ in $G$. In particular, $(|X|+1)$ is a divisor of $q(q-1)$. Since $(\ell q+1)\nmid q(q-1)$, we conclude that $(K\setminus\{1\})\subseteq X$. It follows that $H\cap K=1$ and therefore $H$ is contained in a Frobenius complement of $G$. Without loss of generality, we assume $H\le C$. By Lemma \ref{length}, every  conjugacy class contained in $X\setminus K$ has a unique element belonging to $C$. Therefore we can set $X=(\cup_{i=1}^{s-1}a_i^{G})\cup(K\setminus\{1\})$ where $a_i$ is the unique common element of $C$ and a conjugacy class for all $i\in \{1,\ldots,s-1\}$. In particular, we have $X\cap C=\{a_1,\ldots,a_{s-1}\}$. Since $\{1\}\cup X$ is a left transversal of $H$ in $G$, we have that $\{a_0,a_1,\ldots,a_{s-1}\}$ is a left transversal of $H$ in $C$ for each $a_0\in H$. Therefore the graph $\CS(G,X)$ here is inconsistent with the graphs $\CS(G,X)$ in Theorem \ref{Frobenius}.

Suppose that $\CS(G,Y)$ is a Cayley sum graph of $G$ admitting a subgroup total perfect code $H$. By Lemma \ref{length}, either $(K\setminus\{1\})\subseteq Y$ or $|Y|$ is divisible by $q$. By Lemma \ref{tiff}, $Y$ is a left transversal of $H$ in $G$. In particular, $|Y|$ is a divisor of $q(q-1)$. Note that $|Y|=\ell q-1$ for some nonnegative integer $\ell$ if $(K\setminus\{1\})\subseteq Y$.  Since $(\ell q-1)\nmid q(q-1)$, we conclude that $K\setminus\{1\}$ is not contained in $Y$. Therefore $|Y|$ is divisible by $q$. Since $|Y||H|=|G|=q(q-1)$, we have that $|H|$ is a divisor of $q-1$. Thus $H$ is contained in a Frobenius complement of $G$. Without loss of generality, we assume $H\le C$. Let $a_0$ be the unique common element of $Y$ and $H$. Then $a_0$ is a nonsquare of $H$.  Set $Y\cap C=\{a_0, a_1,\ldots,a_{s-1}\}$. Then $Y=\cup_{i=0}^{s-1}a_i^{G}$. Since $Y$ is a left transversal of $H$ in $G$, we have that $\{a_0,a_1,\ldots,a_{s-1}\}$ is a left transversal of $H$ in $C$.
Therefore the graph $\CS(G,Y)$ here is inconsistent with the graphs $\CS(G,Y)$ in Theorem \ref{Frobenius}.
\end{proof}

\noindent {\textbf{Acknowledgements}}~~This work was supported by the Natural Science Foundation of Chongqing (CSTB2022NSCQ-MSX1054) and the Foundation of Chongqing Normal University (21XLB006).

\medskip
\noindent {\textbf{Data Availability} No data, models, or code were generated or used during the study.
{\small

\begin{thebibliography}{99}
\bibitem{AT2016}
 M. Amooshahi and B. Taeri, On Cayley sum graphs of non-abelian groups, \emph{Graphs Combin.} 32(2016) 17--29.
\bibitem{Big}
N. L. Biggs, Perfect codes in graphs, {\em J. Combin. Theory Ser. B} 15 (1973) 289--296.
\bibitem{CWX2020}
J. Chen, Y. Wang, B. Xia. Characterization of subgroup perfect codes in Cayley graphs. \emph{Discrete Math.}, 343:111813, 2020.
\bibitem{CGW2003}
 B. Cheyne, V. Gupta, C. Wheeler, Hamilton cycles in addition graphs, \emph{Rose Hulman Undergrad. Math. J.} 4(1) (2003), 1--17.
 \bibitem{C1989}
  F.R.K. Chung, Diameters and eigenvalues, \emph{J. Amer. Math. Soc.} 2(2) (1989) 187--196.
 \bibitem{DeS}
I. J. Dejter and O. Serra, Efficient dominating sets in Cayley graphs, {\em Discrete Appl. Math.} 129 (2003) 319--328.
\bibitem{DSLW16}
Y-P. Deng, Y-Q. Sun, Q. Liu and H.-C. Wang, Efficient dominating sets in circulant graphs, {\em Discrete Math.} 340 (2017) 1503--1507.
 \bibitem{FHZ}
R. Feng, H. Huang, and S. Zhou, Perfect codes in circulant graphs, {\em Discret. Math.} 340 (2017) 1522--1527.
\bibitem{GLS2009}
 D. Grynkiewicz, V.F. Lev, O. Serra, Connectivity of addition Cayley graphs, \emph{J. Comin. Theory Ser. B.} 99(2009), 202--217.
\bibitem{HHS}
T. W. Haynes, S. T. Hedetniemi and P. Slater, \textit{Fundamentals of Domination in Graphs}, Marcel Dekker, Inc., New York, 1998.
\bibitem{HXZ18}
H. Huang, B. Xia, and S. Zhou, Perfect codes in Cayley graphs, {\em SIAM J. Discrete Math.} 32 (2018) 548--559.

\bibitem{KS2004}
H. Kurzweil and B. Stellmacher, \emph{The Theory of Finite Groups, An Introduction},
Universitext, Springer, New York-Berlin-Heidelberg, 2004.
\bibitem{Le}
J. Lee, Independent perfect domination sets in Cayley graphs, {\em J. Graph Theory} 37 (2001) 213--219.
\bibitem{L2010}
V.F. Lev, Sums and differences along Hamiltonian cycles, \emph{Discrete Math.} 310(2010) 575--584.
\bibitem{Va75}
J. H. van Lint, A survey of perfect codes, {\em Rocky Mountain J. Math.} 5 (1975) 199--224.
\bibitem{MFW2020}
X. Ma, M. Feng, Min and K. Wang, Subgroup perfect codes in Cayley sum graphs, \emph{Des. Codes Cryptogr.} 88(7) (2020), 1447--1461.
\bibitem{MWY2022}
X. Ma, K. Wang and Y. Yang, Perfect codes in Cayley sum graphs, The electronic journal of combinatorics 29(1) (2022), $\#$P1.21.
\bibitem{MWWZ19}
X. Ma, G. L. Walls, K. Wang, and S. Zhou, Subgroup perfect codes in Cayley graphs,  {\em SIAM J. Discrete Math.} 34 (2020) 1909--1921.
\bibitem{SGS2011}
D. Sinha, P. Garg, A. Singh, Some properties of unitary addition Cayley graphs, \emph{Notes Number Theory and Discrete Math}, 17(3)(2011), 49--59.
\bibitem{Ta13}
T. Tamizh Chelvam and S. Mutharasu, Subgroups as efficient dominating sets in Cayley graphs, {\em Discrete Appl. Math.} 161 (2013) 1187--1190.
\bibitem{Zhou2016}
S. Zhou, Total perfect codes in Cayley graphs, \textit{Des. Codes Cryptogr.} 81 (2016) 489--504.
\bibitem{Z15}
S. Zhou, Cyclotomic graphs and perfect codes, {\em J. Pure Appl. Algebra} 223 (2019) 931--947.
\bibitem{ZZ2021}
 J. Zhang and S. Zhou, On subgroup perfect codes in Cayley graphs, European J. Combin., 91 (2021), 103228. (Corrigenda: https://arxiv.org/abs/2006.11104v2)
\end{thebibliography}
\end{document}